\documentclass[a4paper,reqno]{amsart} %reqno for numbering to the right
\usepackage{amsfonts,amsmath,amssymb,amsfonts}
\usepackage{xcolor}
\usepackage[english]{babel}
\usepackage{fancyhdr}
\usepackage{amsthm}
\usepackage{url}
\usepackage{verbatim}
\usepackage{graphicx}
\usepackage{array}
\usepackage{blkarray}
\usepackage[caption=false]{subfig}
\usepackage{multirow}
\RequirePackage{amsmath,hyperref}%NEW
\RequirePackage[all]{hypcap} [1.11]
\usepackage{enumitem}
\RequirePackage{xr-hyper}[6.00]
\RequirePackage{ifpdf}[2.3]
\RequirePackage{hyperref}[6.83]
   \RequirePackage{breakurl}[1.40]
\RequirePackage[capitalize,nameinlink]{cleveref}[0.19]
\crefname{section}{section}{sections}
\crefname{subsection}{subsection}{subsections}
\Crefname{section}{Section}{Sections}
\Crefname{subsection}{Subsection}{Subsections}
\crefdefaultlabelformat{#2\textup{#1}#3}
\numberwithin{equation}{section}

\usepackage{graphicx}
\DeclareGraphicsExtensions{.jpg,.gif,.eps}
\newcommand{\B}{\mathbf{B}}
\newcommand{\dom}{D}
\newcommand{\Vdom}{V}
\newcommand{\OT}{L}
\newcommand{\LV}{\OT_{\dom}[\Vdom]}
\newcommand{\LU}{\OT_{\dom}[U]}

\newcommand{\LL}{\OT_{\dom}[\dom]}

\newcommand{\Ex}{\mathbb{E}_{x}}
\newcommand{\Ey}{\mathbb{E}_{y}}

\newcommand{\rr}{ {\rm I\!R}}

\newcommand{\rd}{\rr^d}
\newcommand{\Ft}{(\mathcal{F}_t)}
\newcommand*\diff{\mathop{}\!\mathrm{d}}

\newcommand{\norm}[1]{||#1||_{\infty}}

\def\defto{\buildrel def\over =}

\def\half{\frac{1}{2}}

\usepackage{tcolorbox} %to use \begin{tcolorbox} \end{tcolorbox}

\theoremstyle{plain}
\newtheorem{theorem}{Theorem}[section]
\newtheorem{lemma}[theorem]{Lemma}
\newtheorem{proposition}[theorem]{Proposition}

\newtheorem{definition}[theorem]{Definition}
\theoremstyle{remark}
\newtheorem{remark}{Remark}
\theoremstyle{definition}

\begin{document}

  \title[Probabilistic Approach: Semilinear Elliptic Equations]{\textit{A Probabilistic Approach to the Existence of Solutions to Semilinear  Elliptic Equations.}}
  \author{M. E.   Hern\'andez-Hern\'andez}
  \address{University of Leeds, United Kingdom}
  \email{M.E.Hernandez-Hernandez@leeds.ac.uk}

  \author{P.  Padilla-Longoria}
  \address{IIMAS, Universidad Nacional Autonoma de Mexico, UNAM}
   \email{pablo@mym.iimas.unam.mx}
 %  \thanks{}

\begin{abstract}
We study a semilinear elliptic equation with a pure power nonlinearity with exponent $p>1$, and provide sufficient conditions for the existence of positive solutions.  These conditions involve expected exit times from the domain,  $D$, where a solution is defined, and expected occupation times in suitable subdomains of $D$.     They provide an alternative new approach to the geometric or topological sufficient conditions given in the literature for exponents close to the critical Sobolev exponent.
Moreover, unlike standard results, in our probabilistic approach no \emph{a priori}  upper bound restriction is imposed on $p$, which might be supercritical. The proof is based on a fixed point argument using a probabilistic representation formula. We also prove a multiplicity result and discuss possible extensions to the existence of sign changing solutions. Finally, we conjecture that necessary conditions for the existence of solutions might be obtained using a similar probabilistic approach. This motivates a series of natural questions related to the characterisation of topological and geometrical properties of a domain in probabilistic terms.
\end{abstract}


\date{\today}

%\tableofcontents
\maketitle

\section{Introduction}

This paper provides a probabilistic characterisation of sufficient conditions for the existence of  solutions to
\begin{equation}\label{problema}
\left  . \begin{array}{rl}
  \Delta u + \lambda u^p=0, & \text{ in } D,\\ 
  u > 0, & \text{ in } D,\\ 
u = 0, & \text{ on } \partial D,
\end{array}
\right \}  
\end{equation}
 for a given open bounded domain  $D$  in $\rd$, $d > 2$, with a $C^2$  boundary $\partial D$, $\lambda > 0$ and  $p >1$.
 
 Our method relies on rewriting \eqref{problema} as an integral equation in terms of the mathematical expectation of a Brownian motion and, thus, translating  the existence of a solution to \eqref{problema} into a corresponding fixed point problem.  Sufficient conditions for the existence of a fixed point  are then given in terms of  expected exit times and expected occupation times of a Brownian motion on $D$, the region of interest. An advantage of this approach is that it does not require any \emph{a priori} upper-bound restriction on the values of $p$, a strong limitation common to all the analytic approaches or the classical topological methods of the calculus of variations.

Of particular interest is \emph{the critical Sobolev exponent equation}
  \begin{equation}\label{Eq:1}
\left  . \begin{array}{rl}
  \Delta u + u^{p^*-1}=0, & \text{ in } D,\\
u = 0, & \text{ on } \partial D,
\end{array}
\right \}  
\end{equation}
where $p^* = \frac{2d}{d-2}$ is the critical Sobolev exponent for the Sobolev embedding $H_0^{1,2}(D) \subseteq L^p (D)$. 
This equation has been extensively studied and several conditions are known for the existence and multiplicity of solutions (see the discussion below). In fact, it is related to the Yamabe problem, which was solved in the 1980's (see e.g. \cite{Lee1987}, or \cite{Marques2017} for a recent survey on the subject).

 Providing necessary conditions for semilinear elliptic equations with nonlinearities involving powers $p\approx p^*$ (or bigger) remains an open question. Establishing probabilistic conditions for the existence of solutions gives the possibility of exploring necessary conditions for existence from this new perspective. In fact, this approach suggests a series of natural questions related to the probabilistic characterisation 
 of geometric and topological properties of domains (see \cite{Markowsky2011}).

The main contribution of this paper is twofold:
\begin{itemize}
\item [a)] Provide a probabilistic characterisation of  sufficient conditions to  guarantee the existence of nontrivial solutions to \eqref{problema} without the standard restriction $p \le p^*$ (see \cref{T:main}); and conjecture that similar conditions might be also necessary. 

\item [b)] As a particular case, and via localisation arguments, we recover multiplicity  results  for the  critical Sobolev exponent equation \eqref{Eq:1} (see \cref{T:multiplicity}).
\end{itemize} 

Semilinear elliptic equations have  been studied from an analytical perspective   by many authors, including S. I. Pohozaev (1965) \cite{Pohozaev1965},  C. Loewner and  L. Nirenberg (1974) \cite{Nirenberg1974}, Gidas \emph{et. al.} (1979)   \cite{Gidas1979}, Brezis and Veron (1980) \cite{Brezis1980}, Lions (1980) \cite{Lions1980}, V\'eron (1981) \cite{Veron1981}, Coron (1984)  \cite{Coron1984}, Bahri and Coron (1988) \cite{BahriC1988}, and  Ding (1989) \cite{Ding1989}, among the earlier works; and more  recently,  through probabilistic approaches  by Dynkin (1991) \cite{Dynkin1991} and Le-Gall (1999) \cite{LG1999}. 
Specifically for semilinear equations involving the critical Sobolev exponent Brezis and Nirenberg in \cite{BrezisNirenberg} proved existence for certain type of perturbations, independently of the shape of the domain, showing that the compactness that is lost in 
general for this problem can be restored below a certain threshold of the associated functional.
In the analytic context, there is a vast evidence that domain  geometry and  domain topology are  crucial factors for the existence of solutions when the nonlinearity is close to the critical Sobolev exponent. In 1965, S. I. Pohozaev \cite{Pohozaev1965} showed that any solution of  (\ref{problema}) satisfies a variational identity, now known as \emph{the Pohozaev identity}. An immediate consequence thereof establishes the nonexistence of nontrivial bounded solutions to (\ref{problema}) for critical and supercritical nonlinearities, $f$, in convex domains and more generally star-shaped\footnote{Roughly, a star-shaped (or starlike) domain is a domain having at least one point, $P$ for which one can "see" the entire boundary.  In other words, for any point in the domain $Q$ the segment $PQ$ is contained in $D$ } (polar) domains $D$. In fact some extensions to  non-existence results for domains which are non-star-shaped are presented in \cite{Chou1995} (see also some related works \cite{Schaaf2000} and \cite{Dolbeault2010}, references therein).

\subsection{Case $p < p^*$}

The existence of positive solutions for \eqref{problema} does not depend on the shape of the domain $D$ when $1<p < p^*$. Nevertheless, there are multiplicity results depending on the shape of the domain when $p< p^*$ but close to $p^*$, see \cite{Benci1991}, \cite{BenciP1991} and \cite{Benci1994}. On the other hand, if $p \ge p^*$, the existence of positive solutions to \eqref{problema} is  related to the shape of $D$. In particular, if $D$ is a ball in $\rd$, then there are no positive solutions to the Dirichlet problem when $p\ge p^*$, but there are positive solutions if $p < p^*$.

\subsection{Case $p=p^*$}

 \begin{theorem}[Pohozaev] \label{T:Pzv} Let $D$ be a smooth starlike domain, then there are no positive solutions to
 \begin{equation}
\left  . \begin{array}{rl}
  \Delta u + u^{\frac{d+2}{d-2}}=0, & \text{ in } D,\\
u = 0, & \text{ in } \partial D.
\end{array}
\right \}  
\end{equation}
 \end{theorem}
As a special case, Kazdan and Warner (1975)  \cite{KazdanW1975} showed that the geometry of $D$ determines the existence of positive solutions to the equation
 \begin{equation}\label{KW-19}
 -\Delta u = u^6 \text{ in } D, \quad u = 0 \text{ on } \partial D.
  \end{equation}
In particular, \eqref{KW-19} has no positive solution if $D$ is a ball in $\rd$, for $d\ge 3$, whereas if $D$ is an annulus there is a positive solution for all $d$.
More generally, Coron (1984) \cite{Coron1984} showed that the existence of positive solutions to \eqref{Eq:1} holds for  more general domains with a sufficiently small hole.
\begin{theorem} Suppose that $D$ satisfies the following conditions:
\begin{itemize}
\item [i)] $D \supset \{ x \in \rd\,:\, r_1  < |x| < r_2 \}$,
\item [ii)] $\bar{D} \not\supset \{ x \in \rd \,:\, |x| < r_1\}$,
\end{itemize}
for some $0< r_1 < r_2 < \infty$. Then if $r_2 / r_1$ is sufficiently large, problem \eqref{Eq:1} admits a solution in $H^1_0(D)$.
 \end{theorem}

Bahri and Coron (1988) \cite{BahriC1988} proved that \eqref{Eq:1} with  domains having a nontrivial topology admits a solution. Correspondingly, Ding (1989) \cite{Ding1989}  proved that those with trivial topology but which are in some sense a perturbation of a domain with nontrivial topology, also admit a solution.  Existence and multiplicity of solutions have also been studied in contractible domains see, e.g., \cite{Passaseo1989} and \cite{PP1996}.  %(see, e.g., \cite[Theorem 6.2]{PP1996}, and references therein).
 
\subsection{Case: $p > p^*$}  In this case, it is known  that one can find examples of nontrivial domains $D$ in the sense of Bahri and Coron (1988) for which \eqref{problema} has no solution, see, e.g.,  \cite{Passaseo1993, Passaseo1995}.  
 Additionally, \cite{Pasasseo1992}  
 shows that for all $p\ge p^*$ one can find contractible domains where the number of positive solutions to \eqref{problema} is arbitrarily large.

On the other hand, the relationship between  probability theory and  differential equations goes back to  the pioneer  works of Bachelier (1900) \cite{Bachelier01} and Kolmogorov. Later on  Kakutani (1944, 1945) \cite{Kakutani1944, Kakutani1945} showed  the  connection  between Brownian motion and harmonic functions, and Kac (1949, 1951) \cite{Kac49}, \cite{Kac51} established the probabilistic representation of solutions to some partial differential equations. It is well-known, for instance, that the transition probability function for Brownian motion (or Gauss kernel) is the fundamental solution of the heat equation. As for the Dirichlet boundary problem with the Laplacian operator, its  solution -given as a Poisson integral formula  for harmonic functions- admits a probabilistic representation in terms of a Brownian expectation \cite{Chung01}, \cite{Bass01}. In this case, the \emph{Poisson kernel} (or \emph{harmonic measure}) corresponds to the distribution of the \emph{exit distribution} $W_{\tau}$ of the Brownian motion, where $\tau$ is the first exit time from the given domain. 
A sufficient (analytic) condition for the existence (and uniqueness) of a  solution to the Dirichlet boundary problem is  the \emph{Poincar\'e cone condition}. Probabilistic conditions are also known in terms of the probabilistic concept of  \emph{regularity} of the boundary of  $D$ (see \cref{D:Reg} below). 

More recently,  E.B. Dynkin (1991) \cite{Dynkin1991} established connections between the theory of superprocesses and  positive solutions of non-linear partial differential equations involving the operator $Lu - f (u)$, where $L$ is a strongly elliptic differential operator in $\mathbb{R}^d$. An explicit formula for the  solution  for certain type of functions $f$ was given in terms of  hitting probabilities and additive functionals of  the corresponding $(L, f)$-superdiffusion (see S. Watanabe (1968) \cite{Watanabe1968} and Dawson (1975) \cite{Dawson1975}).  

 An $(L,f)$-superdiffusion  is obtained as a \emph{high-density, short-life, small-mass limit} particle systems evolving according to  a Markov  process with generator  $L$ and where   the nonlinear function $f$ describes its branching mechanism.   The admissible  functions $f$ belong to a wide class $\Psi$ of monotone increasing convex functions, which includes the family $f (u) = u^{p}$,  and  $f (x,u) = k(x)u^{p}$, $k > 0$, with the power $p$ restricted to  $1 < p \le 2.$ The solution of the particular case $\Delta u - f (u)$, where  $f(u) =u^p$, $1 < p \le 2$ is given in terms of  the so called \emph{super-Brownian motion} (see also  Dawson (1993) \cite{Dawson1993}).

 In this context, another  outstanding result due to Le-Gall (1999) \cite{LG1999} relies on a Brownian snake approach to construct the exit measure of the super-Brownian motion to get a probabilistic solution to the Dirichlet problem associated with the equation $\Delta u = u^2$ in a regular domain. Some related ideas using the Feynman-Kac formula to study the Dirichlet problem for Schr\"odinger operators $H =\half \Delta + V$ can be found in Aizenman and Simon (1982) \cite{ASimon1982} and references cited therein.

 The probabilistic approach considered here  to solve the Dirichlet problem for the Poisson equation combines the Schauder fixed point theorem with the probabilistic representation of the solution to the Dirichlet equation for the Laplacian operator (see \cref{P:Poisson-f}). This approach has the advantage that there is not an  \emph{a priori} upper-bound restriction on the values of $p$ as in the analytical setting for which the use of a variational approach for $d \ge 3$ imposes the restriction $p \le p^*$, or as in the superdiffusion approach  mentioned above wherein the relation of $p$ to the branching mechanism of the underlying particle systems is meaningful only when $1 < p \le 2$.

The structure of the paper is as follows: Section 2 introduces some notation and recalls the probabilistic representation of the solution to the Poisson equation.  Our main result, \cref{T:main}, is proved in Section 3 using Schauder's fixed point theorem. In Section 4 we discuss some applications and give a multiplicity result. Finally, Section 5 is devoted to our conclusions and open problems. 

\section{Preliminaries}

Let $(\Omega,\mathcal{F}, (\mathcal{F}_t), \mathbb{P})$ be a filtered probability space satisfying the usual conditions. Let $W^x = (W^x_t)_{t\ge 0}$ be a $d-$dimensional $\Ft$-Brownian motion started at $x\in \rd$. For any open  $D \subset \rd$ and $x\in D$,  define the stopping time $\tau_D(W^x)$ as the first exit time of $W^x$ from $D$, that is
\begin{equation}
\tau_{D}^x \equiv \tau_D(W^x) :=\inf\{t\geq 0\, :\, W_t^x \notin D\},
\end{equation}
with the standard convention $\inf  \varnothing = \infty$. We say that $D$ is a \emph{transient domain} whenever  $\mathbb{P} \left [  \tau_D^x  < \infty\right ] =1 $ for all $x\in D$. In particular, any bounded domain is transient. By continuity of the paths of $W^x$, $W^x (\tau^x_{D}) \in \partial D$, i.e.  $\tau_D^x$ is also the first time the process $W^x$ hits the boundary $\partial D$  of $D$.   

As usual, $\mathbb{E}_x$ and $\mathbb{P}_x$ denote, respectively, the expectation and probability with respect to the Brownian motion started at $x \in \rd$; and $\norm{f}$ denotes the sup-norm of a function $f$ on $\rd$.

\begin{remark}\label{R:1}
For any bounded domain $D \subset \rd$, $d\ge 3$,  the Brownian motion is transient  \cite[Theorem 3.26]{MYuval2010} and thus the overall expected occupation time spent in $D$, defined by  $\Ex \int_0^{\infty} \mathbf{1}_{D} \left ( W_t\right)\diff t $, $x\in D$, is finite, which in turn implies that $\sup_{x\in D}\Ex \left [ \tau_D \right ] < \infty$.
\end{remark}
We now recall the probabilistic concept of \emph{regular boundary} as given in \cite[Chapter 4.2, p.245]{Karatzas}, which will play an important role for the continuity of solutions at the boundary.   Intuitively, a point $x\in \partial D$ is regular if the Brownian path started at $x$ exits $\bar{D}$ immediately.   
\begin{definition}\label{D:Reg}
Let $D$ be an open set. A point $x \in \partial D$ is said to be \emph{regular} for $D$ if the first hitting time $\sigma_D^x := \inf \{ t > 0 \,:\, W_t^x \in D^c\}$ satisfies $\mathbb{P} [\sigma_D^x = 0] =1$. 
\end{definition}

It follows \cite[Theorem 4.2.12, p. 245]{Karatzas} that for any $d\ge 2$, any point  $x\in \partial D$ is regular for $D$ if, and only if, for every bounded, measurable function $f :\partial D \mapsto \mathbb{R}$ which is continuous at $x$, one has
\begin{equation}
\lim_{y \to x, y \in D} \Ex f \left ( W_{\tau_D}\right ) = f(x).
\end{equation}

Let us now recall the following well-known result about the existence of solutions to the Poisson equation  and its probabilistic representation (see, e.g., \cite[Chapter 8, Remark 8.7]{MYuval2010}):
\begin{proposition}\label{P:Poisson-f}
Let $D \subset \rd$ be a bounded open domain. If $f \in C_b(D)$ and $u \in C(\bar{D}) \cap C^2(D)$ is a solution of the Poisson problem: $\half \Delta u = f$ in $D$, $u = 0$ on $\partial D$, then $u$ admits the probabilistic representation
   \begin{equation}\label{P-Sol}
u(x) =  \Ex \left [   \int_0^{\tau_{D}} f (W_t) \diff t \right ], \quad x \in D.
\end{equation}
Conversely, if $f$ is H\"older continuous and every $x \in \partial D$ is regular for  $D$, then \eqref{P-Sol} solves the Poisson problem for $f$. 
\end{proposition}

\begin{remark}
The regularity assumption in \cref{P:Poisson-f} for each $x \in \partial D$ guarantees the boundary condition $u(x) = 0$ on $\partial D$. 

 Some classical criteria for regularity are, for instance, that either $D$ satisfies the Poincar\'e cone condition\footnote{A  domain $D \subset \rd$ satisfies the \emph{Poincar\'e cone condition} at $x\in \partial D$ if there exists an $h> 0$ and a cone $V$ based at $x$ with opening angle $\alpha > 0$ such that $B_h (x) \cap V \subset D^c$.}, or that  $\liminf_{r\downarrow 0} \frac{ \mathcal{L} (B_r(x) \cap D^c)}{r^d} >0$ for $x \in D$. Here $\mathcal{L}$ denotes the Lebesgue measure on $\rd$. See also \cite[Chapter 4.2, Section C]{Karatzas} and \cite[Chapter 8.4, Theorem 8.37]{MYuval2010}. 
 \end{remark}

We can now state
the equivalence between the solution to \eqref{problema} and the solution to its corresponding nonlinear integral equation.

\begin{lemma}\label{L:IntEq}
Let $D \subset \rd$ be a bounded open domain and  $u\in C_b^2(D) \cap C(\bar{D})$. Assume  $h: \mathbb{R} \mapsto \mathbb{R}$ is a  continuous function.  If $u$ is a solution to \eqref{problema}, then $u$ solves the (non-linear) integral equation
\begin{equation}\label{EK0}
u(x) =  \Ex \left [   \int_0^{\tau_{D}} h( u (W_t)) \diff t \right ].
\end{equation}
Conversely, if $h$ is $C^1(\mathbb{R})$  and each  $x\in \partial D$ is regular for $D$, then \eqref{EK0} is a solution to \eqref{problema}.
\end{lemma}

\begin{proof}
We only need to verify that $h\circ u : D \mapsto \mathbb{R}$ satisfies the conditions  of  \cref{P:Poisson-f}. Indeed, for the first statement, since $u \in C(\bar{D})$ and $D$ is bounded, the image of $u$ is a bounded interval. Hence, the continuity of $h$ implies that the composition $h\circ u$ is a bounded function. For the second statement, since $h\in C^1(\mathbb{R})$, $\norm{u} <\infty$ and $u \in C^2_b(D)$, then 
$\nabla (h\circ u) = h'(u) \Delta u$ and, further, $|h'(u) \Delta u| \le \sup_{|z| \le \norm{u}} |h'(z)| ||u||_{C^2(D)} \,\, <\,\, \infty$. Thus,   $h\circ u \,\, \in \,\,C^1_b(D)$, which implies then the Lipschitz (and so the H\"older) continuity of $h\circ u$, as required. 
\end{proof}

\subsection{Newtonian potentials and Green functions}

As mentioned before, the connections between probability theory and the theory of PDEs is understood via the analytic concepts of Newtonian potentials and Green functions \cite{Chung01}. In particular, the probabilistic conditions provided in this paper to study the existence of solutions to \eqref{problema} are related to the analytic concept of \emph{Green potentials}\footnote{In fact, the proof of the result given in \cite{BahriC1988} is based on a careful study of the properties of the Green function. See also \cite{Rey1990}.}.  For completeness, we recall this concept in what follows.  

For any suitable function $f$, its \emph{Newtonian potential}, denoted by $Uf$, is the operator defined for each $x$ by
\begin{equation}
Uf(x) \defto c_d \, \int_{\mathbb{R}^d} \frac{f(y)}{|x-y|^{d-2}} \diff y, \quad \quad \text{where} \quad \quad 
c_d \defto \frac{\Gamma (d/2 -1)}{(2\pi)^{d/2}}.
\end{equation}

For each $x$, the mapping $u_x(y)\equiv u(x,y) = \frac{1}{|x-y|^{d-2}}$ is known as the Newton gravitational potential (or also the Coulomb electrostatic potential). Apart from a numerical constant, it represents the potential induced by a mass or charge placed at the point $x$ and evaluated at the point $y$.

It is known \cite[Proposition 3.1, p. 104]{Bass01} that if $f\ge 0$ 
 and $d \ge 3$, then the following equality holds

\[ Uf(x) = \int_0^{\infty} T_t f(x)\diff t,\]
where $T_t$ is the transition operator for the Brownian motion, given by
\[T_t f(x) \defto \mathbf{E}_x \left [  f(X_t) \right ] = \int_{\mathbb{R}^d} p(t;x,y) f(y) \diff y, \]
with $p(t;x,y)$ being the transition function of the Brownian motion 
\[p(t;x,y) = \frac{1}{(2\pi t)^{d/2}} e^{-\frac{|x-y|^2}{2t}}.\]

The connection between the Newtonian potential $Uf$ and the transition operator of the Brownian motion can be seen from the following equality 
\[\int_0^{\infty} p(t;x,y)\diff t = \frac{\Gamma (d/2 -1)}{2 \pi^{d/2}} \frac{1}{|x-y|^{d-2}} \,\,\equiv \,\,c_d \,u(x,y).\]

Moreover, given any domain $D$, the Green potential operator for $D$ is defined, for any suitable $f\ge 0$ by 
\[G_D f(x) \defto \mathbf{E}_x \left [ \int_0^{\tau_D} f(W_t) \diff t \right ]. \]
Note that 
\begin{itemize}
\item [$(i)$] If $D = \mathbb{R}^d$, then $G_{\mathbb{R}^d} f = U f$
\item [$(ii)$]If $f \equiv \mathbb{I}$ ($\mathbb{I}$ denotes the identity function), then 
\[G_D \mathbb{I}(x) = \mathbf{E}_x [\tau_D], \]
\item [$(iii)$]If $f(x) = \mathbf{1}_A(x)$ for a Borel set $A$, then 
\[G_D \mathbf{1}_A(x) = \mathbf{E}_x \left [  \int_0^{\tau_D} \mathbf{1}_B (W_s) \diff s \right ], \]
that is, $G_D \mathbf{1}_A(x)$ is the expected occupation time in $A$ before leaving $D$. 
In fact, one has the following expression
$$G_D \mathbf{1}_A(x)=  G_D(x,A) = \int_A g_D(x,y) \diff y, $$
where $g_D(x,y)$ is called the Green function for $D$.
\item [$(iv)$]If $f =\mathbf{1}_B$ for a Borel set $B$, the integral
\[  \int_0^{\infty} \mathbf{1}_B (X_t)\diff t,\]
represents the total "occupation time" the Brownian motion spends on $B$. Thus, 
$U \mathbf{1}_B (x)$ denotes the corresponding expected occupation time (cf. \cref{R:1}).

\end{itemize}

\section{Main result}
\label{sec:Mresult}

We can now state sufficient conditions to guarantee the existence of a positive solution to the integral equation \eqref{EK0}   for the function $h:  \mathbb{R}_+ \to \mathbb{R}$, $h : y\mapsto  \lambda y^p$, $p>1$, $\lambda >0$. 

Let us first introduce some additional notation. Let $\Vdom $ be an open subset of $\dom$. For each $x\in \bar{\dom}$, set $\LV \equiv G_D \mathbf{1}_V$. That is,  $\LV\,:\, \bar{\dom} \to \mathbb{R}_+ $ by 
\begin{equation}
\LV \,\, :\,\, x \,\, \mapsto\,\, \Ex \left [ \int_0^{\tau_{\dom}} \mathbf{1}_{\{ W_s \in\Vdom \}}\diff s \right ].
\end{equation}

Recall that, for each $x\in \bar{\dom}$, the value $\LV(x)$ gives the expected occupation time  spent by  the Brownian motion $W^x$ (started at $x$) on the subset $\Vdom$ before leaving $\dom$.

\begin{remark}
 Note that, for each $x \in \dom$, $\LV$ is finite  thanks to the transience 
 of the Brownian motion for $d\ge 3$, and further $\LV(x) = 0$ for all $x\in \partial \dom$. 
\end{remark}

\begin{remark}
Observe also that (by monotonicity) for any $U \subset \Vdom \subset \dom$ it follows that
\[ \LU(x )\quad  \,\,\le\,\,  \quad \LV (x), \quad \quad x\in \dom.   \]
In particular, for each $x\in \dom$,
\begin{equation}
\LV(x)\quad \le\quad \LL (x) \quad = \quad \Ex \left [ \int_0^{\tau_{\dom}} \mathbf{1}_{\{ W_s \in \dom \}}\diff s \right ] \quad =\quad \Ex [\tau_{\dom}],
\end{equation}
which in turn implies that
\begin{equation}\label{E:LV-ED}
  \inf_{y\in \Vdom} \LV(y)\quad \le\quad \,\,\, \sup_{x\in \dom} \Ex [\tau_{\dom}].
\end{equation}
\end{remark}

We can now state our  main result.

\begin{theorem} \label{T:main}
Let $d \ge 3$, $p>1$ and $\lambda>0$. Let $\dom \subset \rd$ be a bounded  open domain such that each boundary point $x\in \partial \dom$ is regular for  $\dom$. Suppose that there exists a partition $D_1, D_2\subset D$, with $D_1\subset\subset D$ such that there exist positive constants $m$ and $M$ such that  $m \le M$ satisfying the following conditions
\begin{equation}\label{Cond1}
\sup_{x\in D} \Ex [\tau_{\dom}] \quad \le \quad \frac{M^{1-p}}{\lambda},
\end{equation}

\begin{equation}\label{Cond2}
    \inf_{x\in D_1} \Ex \left [ \int_0^{\tau_{\dom}} \mathbf{1}_{\{ W_s \in D_1 \}}\diff s \right ] \quad \ge \quad \frac{m^{1-p}}{\lambda},
     \end{equation}

\begin{equation}\label{Cond3}
    M\,\, \sup_{x\in D_2} \Ex \left [ \tau_{\dom} \right ] \,\, \left ( \sup_{x\in D_2} \Ex \left [ \int_0^{\tau_{\dom}} \mathbf{1}_{\{ W_s \in D_2 \}}\diff s \right ]  \right )^p \quad \le \quad  \left (\frac{m}{\lambda} \right )^p,
     \end{equation}
     
 then (\ref{problema}) has a positive solution $u \in C_b^2(\dom) \cap C(\overline{\dom})$ such that \begin{equation}\label{C:u}
u\quad \ge \quad m \quad >\quad 0\quad \quad in\quad D_1,  \quad \text{ and } \quad \quad \quad \quad \norm{u}\quad \le \quad M.
\end{equation}
\end{theorem}

\begin{remark} 
Before presenting the proof let us  comment on the assumptions of \cref{T:main}. 
\begin{itemize}
\item [i)] The  inequality in \eqref{Cond2} implies the existence of a subset $D_1 \subset \dom$ for which the  expected occupation time of a Brownian motion in $D_1$ before leaving $\dom$ is bounded below by a positive constant. This guarantees that the solution is nontrivial.
\item [ii)] The inequality in \eqref{Cond3}  imposes a condition between the  exit time from $D$ starting  on $D_2$ (the complement of $D_1$) and the expected occupation time on $D_2$ before leaving $D$. As a matter of fact, this condition ensures that the fixed point will be small in the complement of $D_1$, i.e. it will be concentrated in $D_1$. In other words, that the iteration procedure can be localised (cf. the multiplicity result \cref{T:multiplicity}).
\item [iii)] The above conditions represent, in some sense, a probabilistic way of saying that the domain $D$ is close to having a hole, i.e. that the domain is close to a topologically nontrivial one and they should be compared to previous sufficient conditions for existence of solutions to 
equation \eqref{problema}, e.g. \cite{Coron1984,Ding1989} or the general result \cite{BahriC1988}. Notice that Ding's result makes clear that a purely topological condition cannot be necessary for the existence of solutions, but rather a combination of topological and geometric features of the domain. In this respect, our probabilistic formulation captures these ideas. More generally, the question of how topological or geometric features can be characterised in probabilistic terms, for example, using occupation and exit times of stochastic processes, is a natural and interesting one. In this respect, we refer to \cite{Rajba2011} and \cite{Greg2011} where strong convexity and simple connectedness are studied from a probabilistic point of view.
\end{itemize}
\end{remark}

\begin{proof} (of \cref{T:main})\\
 
 By Lemma \ref{L:IntEq}, proving the existence of a positive solution for the boundary problem  \eqref{problema} is equivalent to solving the nonlinear integral equation \eqref{EK0}. Let us thus rewrite \eqref{problema} as a fixed point problem $u(x) = (Tu)(x)$  for a well-defined  operator $T$ as described below. 

Let $C(\overline{\dom})$ denote the space of real-valued continuous functions on $\overline{\dom}$, endowed with the sup norm $\norm{\cdot}$. Take a partition $D_1, D_2 \subset D$,  and positive constants $M,m >0$ as in the statement.

For any real-valued function $g$ on $\bar{D}$, and any open set $V \subset D$,  define $L_D^g[V]\,:\, \bar{\dom} \to \mathbb{R}_+$ by

\begin{equation}\label{Eq:Lu1}
L_D^g[V]  \,\, :\,\, x \,\, \mapsto\,\, \Ex \left [ \int_0^{\tau_{\dom}} g(W_s) \mathbf{1}_{\{ W_s \in V \}}\diff s \right ].
\end{equation}

Define $\B$ as the set 
\begin{equation} \label{Def:B}
  \B:= \left \{ u \in C(\overline{\dom});\,\,u(x)=0,\,\, x\in \partial \dom : \,\, i)\, \inf_{y \in D_1} L_{D}^u [D_1](y) \ge m,  \quad ii) \,\sup_{y \in D_2} L_{D}^u [D_2](y) \le m, \quad   \text{and} \quad iii)\,\norm{u} \le\, M\right \}.
  \end{equation}
  Notice that this set is nonempty. Indeed, by Urysohn's separation theorem, there exists a continuous function,  $u$, vanishing at the boundary of $D$, with $u\equiv m$ in $D_1$ and $0\le u\le m$ in $D$. It is further necessary to mollify $u$ to guarantee its smoothness.
  
  Observe that $\mathbf{B}$ is convex and a  closed subset of the space $C(\overline{\dom})$. Therefore, $(\B, \norm{\cdot})$ is a complete metric space.

Let $h(y) = \lambda \, y^p$. Define the operator $T$ on $\B$, for each $u\in \B$, as the mapping $Tu : \bar{\dom} \to \mathbb{R}$, 
whose value at $x \in \dom $ is given by
\begin{equation}\label{Def:Tu}
 (T u)(x) :=  \Ex \left [  \int_0^{\tau_{\dom}} h (u(W_t)) \diff t \right ],
 \end{equation}
and $(Tu)(x) = 0$ for $x \in \partial \dom$.

To prove the statement we will rely on the Schauder fixed point theorem. We will  proceed then in two steps:

First we show that the operator $T$  maps $\B$ into $\B$, and that $T$ is a continuous operator. For this, note that if $u\in \B$ then $Tu \in C(\overline{\dom})$.  Moreover,  since $\norm{u} \le M$ and $p > 1$, it follows that for each $x\in \dom$,
\begin{align*}
|Tu (x)| &\le \lambda \, \Ex \left [ \int_0^{\tau_{\dom}} \norm{u}^p \diff t \right ]\,\,
   \le \,\, \lambda \, M^p \sup_{x\in \dom} \Ex [ \tau_{\dom} ] \,\, \le M,
 \end{align*} 
 where the last inequality follows from the equality in  (\ref{Cond1}). 
 Hence, $\norm{Tu} \le M$.

On the other hand, to prove that 
\[\inf_{y \in D_1} L_{D}^{Tu} [D_1](y) \ge m,\] observe that

\begin{align*}
\Ey \left [ \int_0^{\tau_{\dom}} Tu (W_s) \,\mathbf{1}_{\{W_s \in D_1\}} \diff s \right ] &=  \lambda \, \Ey \left [ \int_0^{\tau_{\dom}} \mathbb{E}_{W_s} \left ( \int_0^{\infty}  \, \mathbf{1}_{\{\tau_D > r\}} u^p (W_r) \diff r \right ) \,\mathbf{1}_{\{W_s \in D_1\}} \diff s \right ] \\
&  \ge \lambda \, \Ey \left [ \int_0^{\tau_{\dom}}  \left \{\mathbb{E}_{W_s} \left ( \int_0^{\infty} \mathbf{1}_{\{\tau_D > r\}}\, u (W_r) \diff r \right ) \right \}^p \,\mathbf{1}_{\{W_s \in D_1 \}}  \diff s \right ] \\
& \ge \lambda \, \Ey \left [ \int_0^{\tau_{\dom}}  \inf_{x\in D_1} \left \{\mathbb{E}_{x} \left ( \int_0^{\tau_D} u (W_r) \diff r \right ) \right \}^p \,\mathbf{1}_{\{W_s \in D_1 \}}  \diff s \right ] \\
& \ge \lambda \, \inf_{x\in D_1} \left \{\mathbb{E}_{x} \left ( \int_0^{\tau_D} u (W_r) \diff r \right ) \right \}^p \,\Ey \left [ \int_0^{\tau_{\dom}}   \,\mathbf{1}_{\{W_s \in D_1 \}}  \diff s \right ] \\
& \ge  \lambda \, \inf_{x\in D_1} \left \{\mathbb{E}_{x} \left ( \int_0^{\tau_D} u (W_r) \mathbf{1}_{\{W_r \in D_1 \}}  \diff r \right ) \right \}^p \,\Ey \left [ \int_0^{\tau_{\dom}}   \,\mathbf{1}_{\{W_s \in D_1 \}}  \diff s \right ] \\
& \ge \lambda \, m^p \, \Ey \left [ \int_0^{\tau_{\dom}}  \mathbf{1}_{\{W_s \in D_1 \}}  \diff s \right ] \\
& \ge \lambda \, m^p \, \inf_{y\in D_1}\Ey \left [ \int_0^{\tau_{\dom}}  \mathbf{1}_{\{W_s \in D_1 \}}  \diff s \right ] \ge m
\end{align*}
where we have used Jensen's inequality, inequality $i)$ in the definition of $\mathbf{B}$, and condition \eqref{Cond2}. Taking the infimum over $y \in D_1$ in the previous inequality we obtain the desired result.

It remains to be proved that
\[ \sup_{y \in D_2} L_{D}^{Tu} [D_2](y) \le m.\]
\begin{align*}
\Ey \left [ \int_0^{\tau_{\dom}} Tu (W_s) \,\mathbf{1}_{\{W_s \in D_2\}} \diff s \right ] &=  \lambda \, \Ey \left [ \int_0^{\tau_{\dom}} \mathbb{E}_{W_s} \left ( \int_0^{\infty} \mathbf{1}_{\{\tau_D > r\}} u^p (W_r) \diff r \right ) \,\mathbf{1}_{\{W_s \in D_2\}} \diff s \right ] \\
& \le \lambda \, \Ey \left [ \int_0^{\tau_{\dom}}  \left \{\mathbb{E}_{W_s} \left ( \int_0^{\infty} \mathbf{1}_{\{\tau_D > r\}}\, u (W_r) \diff r \right ) \right \}^{1/p} \,\mathbf{1}_{\{W_s \in D_2 \}}  \diff s \right ] \\
& \le \lambda \, \Ey \left [ \int_0^{\tau_{\dom}}  \left \{ \sup_{x\in D_2}\mathbb{E}_{x} \left ( \int_0^{\tau_D} u (W_r) \diff r \right ) \right \}^{1/p} \,\mathbf{1}_{\{W_s \in D_2 \}}  \diff s \right ] \\
& \le \lambda \, \left \{ \sup_{x\in D_2} \mathbb{E}_{x} \left ( \int_0^{\tau_D} u (W_r) \diff r \right ) \right \}^{1/p} \,\Ey \left [ \int_0^{\tau_{\dom}}   \,\mathbf{1}_{\{W_s \in D_2 \}}  \diff s \right ] \\
& \le \,\lambda \, M^{1/p}  \left \{ \sup_{x\in D_2}\mathbb{E}_{x} \left [ \tau_D \right ] \right \}^{1/p} \,\Ey \left [ \int_0^{\tau_{\dom}}   \,\mathbf{1}_{\{W_s \in D_2 \}}  \diff s \right ] \\
& \le m,
\end{align*}
where we have used Jensen's inequality,  and condition \eqref{Cond3}. Taking the supremum over $y \in D_2$ in the previous inequality we obtain the desired result.
 We can thus conclude that  $T: \B \to \B$, as required.

We now prove that $T$ is continuous. Let $u,v \in \B$ and $x\in \dom$, then  the mean-value theorem yields
\begin{align*}
 |(Tu)(x)  - (Tv)(x)|  &\quad \le\quad  \lambda \, \Ex \left [ \int_0^{\tau_{\dom}} \Big |  u^p (W_t)  - v^p (W_t) \Big |  \diff t \right   ]  \\
 & \quad \le \quad \,K\, \mathbb{E}_x \left [ \int_0^{\tau_{\dom}}   \Big |  u (W_t)  - v (W_t) \Big | \diff t \right   ], 
\end{align*}
where $K := \lambda \, p M^{p-1}$, since $u^p$ is smooth and we can take the Lipschitz  constant as the maximum of its derivative. Hence,
\begin{align*}
 \norm{T (u)  - T (v)} 
 & \quad \le \quad  \lambda \, p M^{p-1} \norm{u - v } \sup_{x\in \dom}\Ex  \left [\tau_{\dom}   \right   ]. 
\end{align*}
By \cref{R:1}, we obtain that $ \norm{T(u)-T(v)} \le  C \, \norm{u-v}$, for all $u,v \in \B$, where $C = \lambda \, p M^{p-1}\sup_{x\in \dom}\Ex  \left [\tau_{\dom}   \right   ] $ is a finite constant independent of $u$ and $v$. Hence, $T$ is continuous. In fact, by the Arzela-Ascoli theorem, $T(\textbf{B})$ is precompact.  Therefore, by Schauder's fixed point theorem\footnote{
Schauder's Fixed point theorem:
\emph{Let $X$ be a Banach space, $M\subset X$ be  non-empty, convex, and closed, and $T: M \subset X \mapsto  M$ be a continuous operator such that $T(M)$ is precompact. Then $T$ has a fixed point.} }, $T$ has a fixed point $u_* \in \B$, as required.

\end{proof}

\section{Applications}
In what follow we present two examples. In the first, we consider the case of the ball and $p>2$, which includes most values of the critical Sobolev exponent ($d>6$).
We show that, at least for a sufficiently large radius, the hypotheses of \cref{T:main} cannot hold, since by Pohozaev's identity the only solution is trivially identically equal to zero. 
In the second example for the case of an annulus it is shown that the conditions of the main theorem hold, provided the inner and outer radii are suitably chosen. This recovers Coron's result (see \cite{Coron1984}) and shows that our result is nonempty.

\textbf{Example 1.} Consider the case where the domain is  $D = B_{T}(0) \subset \rd$, $p>2$, and $d\ge 3$. It is known that in this case, the only solution to \eqref{problema} is the trivial one. Hence, we expect that, for consistency, at least one of the conditions in \cref{T:main} would not be satisfied. Indeed, at least for $T$ sufficiently large we can prove that \eqref{Cond1} and \eqref{Cond2} hold, but \eqref{Cond3} does not. 

Let $ 0 <r < R$. Take $D_1 \equiv B_R \setminus B_r$ and $D_2 \equiv D_1^c$. It is not difficult to see that \eqref{Cond1} holds with $M := \left( \lambda \, T^2 /d \right )^{\frac{1}{1-p}}$; and by the Krylov-Safanov inequality, there exists some $K(T,R,r,d)>0$ such that \eqref{Cond2} holds with $m  = (\lambda \, K)^{\frac{1}{1-p}}$. 

Now suppose that \eqref{Cond3} holds, then note that
\begin{align*}
 \left ( \frac{m}{\lambda} \right )^p &\ge M\,\, \sup_{x\in D_2} \Ex \left [ \tau_{\dom} \right ] \,\, \left ( \sup_{x\in D_2} \Ex \left [ \int_0^{\tau_{\dom}} \mathbf{1}_{\{ W_s \in D_2 \}}\diff s \right ]  \right )^p \\ &=    \frac{M^{2-p}}{\lambda}  \, \left ( \sup_{x\in D_2} \Ex \left [ \int_0^{\tau_{\dom}} \mathbf{1}_{\{ W_s \in D_2 \}}\diff s \right ]  \right )^p  \\
 &\ge  \frac{M^{2-p}}{\lambda}  \left ( \inf_{x\in D_1} \Ex \left [ \int_0^{\tau_{\dom}} \mathbf{1}_{\{ W_s \in D_1 \}}\diff s \right ] \right )^p\\
 &\ge \frac{M^{2-p}}{\lambda} \left ( \frac{m^{1-p}}{\lambda} \right)^p. 
\end{align*}
Hence,
\begin{equation}
m^p \ge \frac{M^{2-p}}{\lambda} m^{p-p^2}
\end{equation}
and thus
\begin{equation}
m^{p^2}  \ge \left (\frac{\lambda \,T^2}{d}\right)^{\frac{2-p}{1-p}}\frac{1}{\lambda} = \left (\frac{T^2}{d}\right)^{\frac{2-p}{1-p}}\lambda^{\frac{1}{1-p}}.
\end{equation}
Therefore, since $p>2$ and $d\ge 3$, the inequality above does not hold if one takes $T$ sufficiently large.  \\

Notice that the estimates we are providing are certainly not optimal. A careful study with $p$ equal to the critical Sobolev exponent should provide a proof for any ball, independently of its size. Moreover, in principle, by Pohozhaev's identity, these consideration should extend to any starshaped domain.

\textbf{Example 2.} Take positive constants  $\delta, r, R$ and $T$   satisfying $0 < \delta < r < R < T$. Also, let us fix $1<p<\sqrt{2}$ and $d=3$. 
Consider now the domain given by the annulus $D \equiv A(\delta, T) \defto B_T \setminus B_{\delta}$, and take the partition $D_1$ and $D_2$ as in Example 1. Besides, we will require that $D_2$ contains all the $x$ such that
$$
|x|= (\delta T (\delta+ T)/2)^{1/3}.
$$
Observe that $D_1 = A(r,R)$ and $D_2 = A(\delta,r) \cup A(R,T)$. We need to show that the conditions \eqref{Cond1}-\eqref{Cond3} in \cref{T:main} hold for this case. 

Indeed, to obtain $M$ in condition \eqref{Cond1}, 
one can prove that\footnote{In particular, note that when $\delta \to 0$, then we recover the result for the ball $B_T(0)$.}

\begin{align}\label{Ex-A3}
\mathbb{E}_{x} (\tau_{A(\delta,T)}) 
&=\frac{1}{3}\left (  T^2 + T\delta + \delta^2 -\frac{\delta T (\delta +T)}{|x|}  - |x|^2  \right ), \quad
\quad x \in A(\delta,T). 
\end{align}

Define, for $z \in [r,R]$, the mapping $ f(z) = \frac{1}{3}\left (  T^2 + T\delta  + \delta ^2 -\frac{\delta T (\delta +T)}{z}  - z^2  \right ) $. Then
\begin{equation}\label{z-opt}
f'(z) = \frac{1}{3}\left ( \frac{\delta T (T+ \delta )}{z^2} - 2z \right ) = 0\quad  \text{ if and only if } \quad  z^3 = \frac{\delta T (T+\delta )}{2}.
\end{equation}

A direct calculation shows that this condition determines the maximum of $f$ and so we obtain
\begin{equation}
\sup_{x\in A(\delta,T)} \mathbf{E} [\tau_{A(\delta,T)}] = \frac{T^2 + T\delta + \delta^2}{3} - \frac{1}{3}\left ( \frac{\delta T (T+ \delta)}{2}\right)^{2/3}.
\end{equation}
Hence, condition \eqref{Cond1} holds with 
\begin{align*}
M \equiv M(\delta,T,p) &= \left ( \frac{T^2 + T\delta + \delta^2}{3} - \frac{1}{3}\left ( \frac{\delta T (T+ \delta)}{2}\right)^{2/3} \right )^{\frac{1}{1-p}}. 
\end{align*}

As before, the second condition in \eqref{Cond2} holds with $m= K(\delta,r,R,T)^{\frac{1}{1-p}}$ for some positive constant $K$ thanks to the Krylov-Safanov inequality. 

It remains to prove that \eqref{Cond3} holds.

Since $D_2$ contains the points $x \in A(\delta,T)$ for which the supremum in \eqref{Ex-A3} is achieved, it follows that 
\begin{align*}
\sup_{x\in D_2} \mathbf{E} [\tau_D]&= M^{1-p}.
\end{align*}
Therefore, using that

\begin{align*}
    M\,\, \sup_{x\in D_2} \Ex \left [ \tau_{\dom} \right ] \,\, \left ( \sup_{x\in D_2} \Ex \left [ \int_0^{\tau_{\dom}} \mathbf{1}_{\{ W_s \in D_2 \}}\diff s \right ]  \right )^p \,\, &\le M\,\, \left (\sup_{x\in D_2} \Ex \left [ \tau_{\dom} \right ] \right )^{p+1},  
\end{align*}
it follows that
\begin{align*}M^{(1-p)(p+1) + 1} = M^{2-p^2} =  \left ( \frac{T^2 + T\delta + \delta^2}{3} -\frac{1}{3} \left ( \frac{\delta T (T+ \delta)}{2}\right)^{2/3} \right )^{\frac{2-p^2}{1-p}} < (T^2 + T\delta + \delta^2 )^{\frac{2-p^2}{1-p}} \le m^p,
\end{align*}
which holds for $T$ sufficiently large and $p\in (1,\sqrt{2})$.

%%%%%%%%%%%%%%%%%%%%%
As an immediate extension we obtain the following multiplicity result:

\begin{theorem}\label{T:multiplicity}
Let $d \ge 3$ and $p>1$. Let $\dom \subset \rd$ be a bounded  open domain such that each $x\in \partial \dom$ is regular for  $\dom$. Assume that there is a subset $V$ with $s$ connected components $V_1,\ldots, V_s$ in $D$. If for each partition $D_1^i \equiv V_i$ and $D_2^i \equiv V_i^c$, $i=1, \ldots s$,  we can find corresponding $m_k$, $M_k$   satisfying the  inequalities  \eqref{Cond1}-\eqref{Cond3}, then there are at least $2^s-1$ distinct solutions to problem \eqref{Eq:1}.
\end{theorem}

\begin{proof}
The proof of this theorem follows directly by applying a similar  argument to the one used in the proof of 
\cref{T:main}.  
 Note that, excluding the empty set (which should not be taken into account, since it corresponds to the zero solution),  there are N:=$2^{s}-1$ possible open subsets $\hat{V}_k  \subset \dom$, $k =1,\ldots, N$, corresponding to a fixed choice of connected components. Define
\[ I_k := \textit{set of indices $i \in \{1,\ldots, s\}$ such that $V_i \in \hat{V}_k$.}\]
For each subset $\hat{V}_k$ we can thus define the partition $\hat{D}_1^k \equiv \hat{V}_k$ and $\hat{D}_2^k= D \setminus \hat{V}_k $. Define the  constants 

\[\hat{m}_k := \max \{ m_i :  i \in I_k\}\quad \quad 
and \quad \quad 
\hat{M}_k := \min \{ M_i :  i \in I_k\}.\]
A fixed point argument can now be used: for each $k=1, \ldots, N$, define the operator $T_k$ as in \eqref{Def:Tu} but now  defined on the test functions 

\begin{align}
      \B_k:=  \Big \{ u \in C(\overline{\dom}),\,\, u(x)=0,\,\, & x\in \partial \dom : \,\, i)\, \inf_{y \in \hat{D}_1^k} L_{D}^u [\hat{D}_1^k](y) \ge \hat{m}_k, \nonumber \\   
      & \quad \quad \quad \quad ii) \,\sup_{y \in \hat{D}_2^k} L_{D}^u [\hat{D}_2^k](y) \le \hat{m}_k,    \quad   \text{and} \quad iii)\,\norm{u} \le\, \hat{M}_k \Big \}. \nonumber
  \end{align}
The proof follows from a similar argument as before implying the existence of a nontrivial fixed point in $\hat{V}_k$, for each $k$, and therefore, excluding the trivial case, there are, at least, $2^s-1$ solutions, as required.
\end{proof}

\section{Conclusions}

In this work we present a probabilistic approach to study the existence of solutions to a type of elliptic equations. We highlight the following relevant observations:

In probabilistic terms, we provided sufficient conditions to guarantee the existence of positive solutions to \eqref{problema}. 
Indeed, the hypotheses on the domain $D$ in Theorem \ref{T:main} can be viewed as a probabilistic characterization of the existence of a  "hole" in $D$. Intuitively, 
it should be contrasted with the sufficient conditions on the topology of $D$ given by Bahri and Coron in  \cite{Coron1984,BahriC1988}, where it is shown that if there is a hole, then a nontrivial solution exists and the result by Ding (\cite{Ding1989} in which a "quasihole" is defined and it is proved that it is a sufficient condition for existence. In the same spirit the multiplicity results in \cite{Passaseo1989}
and \cite{PP1996} can be phrased as guaranteeing as many solutions as the number of  holes (or quasiholes) of the domain. 
\cref{T:main} shows that there exists a nontrivial connection between the topology/geometry of the underlying domain $D$ and expected exit and occupation times of the Brownian motion. This approach should be contrasted with the use of other 
methodological analytical tools such as potential theory and Green's functions.

It is natural to conjecture that conditions similar to the ones given in the main theorem might be also necessary for the existence of a nontrivial solution. Our approach also suggests that there has to be a natural characterisation of topological and geometrical properties of a domain in  probabilistic terms as it is known for \cite{Greg2011,Rajba2011}, where connectedness and simple connectedness are studied. This is reminiscent of the conditions on the capacity of the boundary of the domain of elliptic equations for boundary conditions to be achieved, as well as the probabilistic characterisation of regularity of a boundary point (analytically verified through the Poincare's cone condition) in terms of expected exit times.
There are several natural extensions of the ideas here presented. In the first place, the study of sign-changing solutions
and  multiplicity results along the same lines as \cref{T:multiplicity}  can be dealt with using these probabilistic methods. Additionally, it is possible to introduce a potential in equation \ref{problema} as in a nonlinear Schr\"odinger equation, for which similar probabilistic representation formulas to the ones we consider exist.
Finally, it would be interesting to consider different kinds of differential operators, other than Laplacian, for instance, more general elliptic operators involving different stochastic processes, e.g. the fractional Laplacian and  
operators associated to L\'evy processes. It is natural to conjecture that the conditions will be given in terms of the exit and occupation times of the corresponding processes.

%\section*{Acknowledgements}

\bibliographystyle{plain}
\bibliography{bibfile16Aug2020}

\end{document}